\documentclass{amsart}
\usepackage{amsmath}
  \usepackage{paralist}
  \usepackage{graphics} 
  \usepackage{epsfig} 
\usepackage{graphicx}  \usepackage{epstopdf}
 \usepackage[colorlinks=true]{hyperref}
\hypersetup{urlcolor=blue, citecolor=red}

\newtheorem{theorem}{Theorem}[section]
\newtheorem{corollary}{Corollary}

\newtheorem{proposition}{Proposition}

\theoremstyle{definition}
\newtheorem{definition}[theorem]{Definition}
\newtheorem{remark}{Remark}

\newcommand{\RR}{\mathbb R}
\newcommand{\NN}{\mathbb N}

\allowdisplaybreaks

\title[Nonlinear Robin problems] 
      {Robin problems with indefinite linear part and competition phenomena}

\author[N.S. Papageorgiou, V.D. R\u adulescu and D.D. Repov\v s]{}


\begin{document}
\maketitle

\centerline{\scshape Nikolaos S. Papageorgiou}
\medskip
{\footnotesize
 \centerline{Department of Mathematics, National Technical University}
   \centerline{Zografou Campus, Athens 15780, Greece}
} 

\medskip

\centerline{\scshape Vicen\c tiu D. R\u adulescu}
\medskip
{\footnotesize
 \centerline{ Department of Mathematics, Faculty of Sciences,}
   \centerline{King Abdulaziz University, P.O.
Box 80203, Jeddah 21589, Saudi Arabia} 
}

\medskip
\centerline{\scshape Du\v san D. Repov\v s}
\medskip
{\footnotesize
 \centerline{Faculty of Education and Faculty of Mathematics and Physics,}
   \centerline{University of Ljubljana, Kardeljeva plo\v{s}\v{c}ad 16, SI-1000 Ljubljana, Slovenia}
}

\bigskip

\begin{abstract}
We consider a parametric semilinear Robin problem driven by the Laplacian plus an indefinite potential. The reaction term involves competing nonlinearities. More precisely, it is the sum of a parametric sublinear (concave) term and a superlinear (convex) term. The superlinearity is not expressed via the Ambrosetti-Rabinowitz condition. Instead, a more general hypothesis is used. We prove a bifurcation-type theorem describing the set of positive solutions as the parameter $\lambda > 0$ varies. We also show the existence of a minimal positive solution $\tilde{u}_\lambda$ and determine the monotonicity and continuity properties of the map $\lambda \mapsto \tilde{u}_\lambda$.
\end{abstract}

\section{Introduction}

Let $\Omega \subseteq \RR^N $ ($N\geq 2$) be a bounded domain with a $C^2$-boundary $\partial\Omega$. In this paper we study the following parametric Robin problem
\begin{equation}\tag{$P_\lambda$}\label{eqp}
	\left\{\begin{array}{ll}
	-\Delta u(z)+\xi(z)u(z)=\lambda g(z,u(z))+f(z,u(z))\ \mbox{in}\ \Omega \\
\displaystyle	\frac{\partial u}{\partial n}+\beta(z)u=0\ \mbox{on}\ \partial \Omega.
	\end{array}\right\}
\end{equation}

In this problem, $\lambda > 0$ is a parameter, $\xi \in L^s (\Omega)$ ($s>N$) is a potential function which is indefinite (that is, sign changing) and in the reaction, $g(z,x)$ and $f(z,x)$ are Carath\'eodory functions (that is, for all $x\in\RR$, $z\mapsto g(z,x),f(z,x)$ are measurable and for almost all $z\in \Omega$, $x\mapsto g(z,x),f(z,x)$ are continuous). We assume that for almost all $z\in \Omega$, $g(z,\cdot)$ is strictly sublinear near $+\infty$ (concave nonlinearity), while for almost all $z\in \Omega$, $f(z,\cdot)$ is strictly superlinear near $+\infty$ (convex nonlinearity). Therefore the reaction in problem (\ref{eqp}) exhibits the combined effects of competing nonlinearities (``concave-convex problem"). The study of such problems was initiated with the well-known work of Ambrosetti, Brezis and Cerami \cite{2}, who dealt with a Dirichlet problem with zero potential (that is, $\xi \equiv 0$) and the reaction had the form
$$\lambda x^{q-1} + x^{r-1}\ \mbox{for all}\ x\geq 0\ \mbox{with}\ 1<q<2<r<2^*.$$

They proved a bifurcation-type result for small values of the parameter $\lambda > 0$. The work of Ambrosetti, Brezis and Cerami \cite{2} was extended to more general classes of Dirichlet problems with zero potential by Bartsch and Willem \cite{4}, Li, Wu and Zhou \cite{8}, and R\u adulescu and Repov\v s \cite{15}.

Our aim in this paper is to extend all the aforementioned results to the more general problem (\ref{eqp}). Note that when $\beta \equiv 0$, we recover the Neumann problem with an indefinite potential. Robin and Neumann problems are in principle more difficult to deal with, due to the failure of the Poincar\'e inequality. Therefore in our problem, the differential operator (left-hand side of the equation) is not coercive (unless $\xi \geq 0$, $\xi \not\equiv 0$). Recently we have examined Robin and Neumann problems with indefinite linear part. We mention the works of Papageorgiou and R\u adulescu \cite{12, 13,14}. In \cite{12} the problem is parametric with competing nonlinearities. The concave term is $-\lambda |x|^{q-2}x$, $1<q<2$, $x\in \RR$ (so it enters into the equation with a negative sign) while the perturbation $f(z,x)$ is Carath\'eodory, asymptotically linear near $\pm \infty$ and resonant with respect to the principal eigenvalue. We proved a multiplicity result for all small values of the parameter $\lambda > 0$, producing five nontrivial smooth solutions, four of which have constant sign (two positive and two negative).

In this paper, using variational tools together with truncation, perturbation and comparison techniques, we prove a bifurcation-type theorem, describing the existence and multiplicity of positive solutions as the parameter $\lambda > 0$ varies. We also establish the existence of a minimal positive solution $\tilde{u}_\lambda$ and determine the monotonicity and continuity properties of the map $\lambda \mapsto \tilde{u}_\lambda$.

\section{Preliminaries}

Let $X$ be a Banach space and $X^*$  its topological dual. By $\left\langle \cdot,\cdot\right\rangle$ we denote the duality brackets for the dual pair $(X^*,X)$. Given $\varphi\in C^1(X,\RR)$, we say that $\varphi$ satisfies the ``Cerami condition" (the ``C-condition" for short), if the following property is satisfied:
\begin{center}
``Every sequence $\{u_n\}_{n\geq 1}\subseteq X$ such that $\{\varphi(u_n)\}_{n\geq 1}\subseteq\RR$ is bounded and
$$(1+||u_n||)\varphi'(u_n)\rightarrow 0\ \mbox{in}\ X^*\ \mbox{as}\ n\rightarrow\infty,$$
admits a strongly convergent subsequence''.
\end{center}

This is a compactness-type condition on the functional $\varphi(\cdot)$. It leads to a deformation theorem from which one can derive the minimax theory for the critical values of $\varphi$ (see, for example, Gasinski and Papageorgiou \cite{6}). The following notion is central to this theory.

\begin{definition}\label{def1}
	Let $Y$ be a Hausdorff topological space and $E_0 , E,D \subseteq Y$  nonempty, closed sets such that $E_0 \subseteq E$. We say that the pair $\{E_0 , E\}$ is linking with $D$ in $Y$ if:
	\begin{itemize}
		\item[(a)] $E_0 \cap D = \emptyset $;
		\item[(b)] For any $\gamma \in C(E,Y)$ such that $\gamma \left|_{E_0} = {\rm id} \right|_{E_0}$, we have $\gamma(E)\cap D \neq \emptyset$.
	\end{itemize}
\end{definition}

	Using this topological notion, one can prove the following general minimax principle, known in the literature as the ``linking theorem" (see, for example, Gasinski and Papageorgiou \cite[p. 644]{6}).
\begin{theorem}\label{th2}
	Assume that $X$ is a Banach space, $E_0, E, D \subseteq X$ are nonempty, closed subsets such that $\{E_0, E\}$ is linking with $D$ in $X$, $\varphi \in C^1 (X,\RR)$ satisfies the $C$-condition
	$$\sup\limits_{E_0}\varphi < \inf\limits_{D} \varphi$$
	and $c=\inf\limits_{\gamma\in\Gamma}\sup\limits_{u\in E}\varphi(\gamma(u))$, where $\Gamma=\left\{\gamma \in C(E,X):\gamma\left|_{E_0}={\rm id}\right|_{E_0}\right\}$. Then $c \geq\inf\limits_{D}\varphi$ and $c$ is a critical value of $\varphi$ (that is, there exists $u\in X$ such that $\varphi' (u)=0,\ \varphi(u)=c$).
\end{theorem}

With a suitable choice of the linking sets, we can produce as corollaries of Theorem \ref{th2}, the main minimax theorems of the critical point theory. For future use, we recall the so-called ``mountain pass theorem".
\begin{theorem}\label{th3}
	Assume that $X$ is a Banach space, $\varphi\in C^1 (X,\RR)$ satisfies the $C$-condition, $u_0,\ u_1\in X,\ ||u_1-u_0||>r>0$,
	$$\max\left\{\varphi(u_0),\varphi(u_1)\right\}<\inf \left[ \varphi(u):||u-u_0||=\rho \right]=m_\rho$$
	and $c=\inf\limits_{\gamma\in\Gamma}\max\limits_{0\leq t\leq 1}\varphi ( \gamma(t))$ with $\Gamma = \left\{\gamma \in C ([0,1],X):\gamma(0)=u_0, \gamma(1)=u_1\right\}$. Then $c\geq m_\rho$ and $c$ is a critical value of $\varphi$.
\end{theorem}
\begin{remark}
Theorem \ref{th3} can be deduced from Theorem \ref{th2} if we have
$E_0 = \left\{ u_0,u_1\right\}$, $E= \left\{ u=(1-t)u_0 + tu_1,\ 0\leq t \leq 1 \right\}$, $D=\partial B_\rho (u_0)=\{ u\in X:||u-u_0||$ $=\rho \}.$	
\end{remark}

In the analysis of problem (\ref{eqp}), we will use the following spaces: the Sobolev space $H^1(\Omega)$, the Banach space $C^1(\overline{\Omega})$ and the boundary Lebesgue spaces $L^r(\partial\Omega)$, $1\leq r \leq \infty$.

By $||\cdot||$ we denote the norm of the Sobolev space $H^1(\Omega)$. So
$$||u||=\left[ ||u||^2_2 + ||Du||^2_2 \right]^{\frac{1}{2}}\ \mbox{for all}\ u\in H^1(\Omega).$$

The  space $C^1(\overline{\Omega})$ is an ordered Banach space with positive cone
$$C_+ = \left\{ u\in C^1(\overline{\Omega}):\ u(z)\geq 0\ \mbox{for all}\ z\in \overline{\Omega} \right\}.$$

We will use the open set $D_+ \subseteq C_+$ defined by
$$D_+=\left\{u\in C_+ :u(z)>0\ \mbox{for all}\ z\in \overline{\Omega}\right\}.$$

On $\partial\Omega$ we consider the $(N-1)$-dimensional Hausdorff (surface) measure $\sigma(\cdot).$

Using this measure, we can define the Lebesgue spaces $L^r(\partial\Omega)$ ($1\leq r\leq \infty$) in the usual way. Recall that the theory of Sobolev spaces says that there exists a unique continuous linear map $\gamma_0 : H^1 (\Omega)\rightarrow L^2 (\partial\Omega)$, known as the ``trace map", such that
$$\gamma_0 (u)=u|_{\partial\Omega}\ \mbox{for all}\ u\in H^1(\Omega)\cap C(\overline{\Omega}).$$

This map is not surjective and it is compact into $L^q(\partial\Omega)$ if $1\leq q < \frac{2(N-1)}{N-2}\ if\ N \geq 3$ and into $L^q(\partial\Omega)\ \mbox{for all}\ q\geq 1\ \mbox{if}\ N=1,2.$

In what follows, for the sake of notational simplicity, we drop the use of the map $\gamma_0$. All restrictions of Sobolev functions on $\partial\Omega$ are understood in the sense of traces.

Let $f_0:\Omega \times \RR \rightarrow \RR$ be a Carath\'eodory function such that
$$\left| f_0(z,x) \right|\leq a_0 (z)(1+|x|^{r-1})\ \mbox{for almost all}\ z\in\Omega\ \mbox{and all}\ x\in\RR,$$
with $a_0\in L^\infty(\Omega), 1<r<2^*=\left\{\begin{array}{ll}
		\frac{2N}{N-2}&\mbox{if}\ N\geq 3\\
		+\infty&\mbox{if}\ N=1,2.
	\end{array}\right.  $

We set $F_0(z,x)=\int^x_0 f_0(z,s)ds$. Also, let $\xi\in L^s(\Omega)$ $(s>N)$ and $\beta\in W^{1,\infty}(\partial\Omega)$ with $\beta(z)\geq 0$ on $\partial\Omega$. We consider the $C'$-functional $\varphi_0 : H^1(\Omega)\rightarrow \RR$ defined by
$$\varphi_0(u)=\frac{1}{2}\vartheta(u)-\int_\Omega F_0(z,u)dz,$$
where
$$\vartheta(u)=||Du||^2_2+\int_\Omega \xi(z)u^2 dz+\int_{\partial\Omega}\beta(z)u^2 d\sigma\ \mbox{for all}\ u\in H^1(\Omega).$$

The next result follows from Papageorgiou and R\u adulescu \cite[Proposition 3]{11} using the regularity theory of Wang \cite{16}.
\begin{proposition}\label{prop4}
	Let $u_0 \in H^1(\Omega)$ be a local $C^1(\overline{\Omega})$-minimizer of $\varphi_0$, that is, there exists $\rho_0 > 0$ such that
	$$\varphi_0 (u_0)\leq \varphi_0 (u_0+h)\ \mbox{for all}\ h\in C^1(\overline{\Omega})\ \mbox{with}\ ||h||_{C^1(\overline{\Omega})}\leq \rho_0.$$
	Then $u_0\in C^{1,\alpha}(\overline{\Omega})$ with $\alpha \in (0,1)$ and $u_0$ is also a local $H^1(\Omega)$-minimizer of $\varphi_0$, that is, there exists $\rho_1 > 0$ such that
	$$\varphi_0(u_0)\leq \varphi_0 (u_0+h)\ \mbox{for all}\ h\in H^1(\Omega)\ \mbox{with}\ ||h||\leq \rho_1.$$
\end{proposition}
	
We will need some facts concerning the spectrum of $-\Delta$ with Robin boundary condition. Details can be found in Papageorgiou and R\u adulescu \cite{11, 14}.

So, we consider the following linear eigenvalue problem
\begin{equation}\label{eq1}
	-\Delta u(z)+\xi(z)u(z)=\hat{\lambda}u(z)\ \mbox{in}\ \Omega,\ \frac{\partial u}{\partial n} + \beta (z)u=0\ \mbox{on}\ \partial\Omega.
\end{equation}

We know that there exists $\mu > 0$ such that
\begin{equation}\label{eq2}
	\vartheta(u)+\mu ||u||^2_2 \geq c_0 ||u||^2\ \mbox{for all}\ u\in H^1(\Omega)\ \mbox{and for some}\ c_0>0.
\end{equation}

Using (\ref{eq2}) and the spectral theorem for compact self-adjoint operators, we generate the spectrum of (\ref{eq1}), which consists of a strictly increasing sequence $\{ \hat{\lambda}_k \}_{k \geq 1} \subseteq \RR$ such that $\hat{\lambda}_k\rightarrow + \infty$. Also, there is a corresponding sequence $\left\{ \hat{u}_n \right\}_{n \geq 1} \subseteq H^1(\Omega)$ of eigenfunctions which form an orthonormal basis of $H^1(\Omega)$ and an orthogonal basis of $L^2(\Omega)$. In fact, the regularity theory of Wang \cite{16} implies that $\left\{ \hat{u}_n \right\}_{n \geq 1} \subseteq C^1(\overline{\Omega})$. By $E(\hat{\lambda}_k)$ (for every $k\in \NN$) we denote the eigenspace corresponding to the eigenvalue $\hat{\lambda}_k,k\in \NN$. We have the following orthogonal direct sum decomposition
$$H^1(\Omega)=\overline{ \mathop{\oplus}\limits_{k \geq 1} E(\hat{\lambda}_k)}.$$
	
Each eigenspace $E(\hat{\lambda}_k)$ has the so-called ``unique continuation property" (UCP for short) which says that if $u\in E(\hat{\lambda}_k)$ vanishes on a set of positive Lebesgue measure, then $u=0$. The eigenvalues $\{ \hat{\lambda}_k \}_{k\geq 1}$ have the following properties:
\begin{eqnarray}\label{eq3}
	&\bullet & \hat{\lambda}_1\ \mbox{is simple (that is},\ \dim E(\hat{\lambda}_1)=1);\nonumber \\
	&\bullet & \hat{\lambda}_1 =\inf\left[ \frac{\vartheta(u)}{||u||^2_2}:u\in H^1(\Omega),u\neq 0 \right]; \\
	&\bullet & \mbox{for}\ m \geq 2\ \mbox{we have} \nonumber
\end{eqnarray}
\begin{align}\label{eq4}
	\hat{\lambda}_m  = & \sup\left[ \frac{\vartheta(u)}{||u||^2_2}: u \in \mathrel{\mathop{\oplus}^{m}_{k=1}} E(\hat{\lambda}_k), u\neq 0 \right] \nonumber \\
	 				 = & \inf\left[ \frac{\vartheta(u)}{||u||^2_2}: u \in \overline{\mathop{\oplus}\limits_{k \geq m} E(\hat{\lambda}_k)}, u\neq 0 \right]
\end{align}

In (\ref{eq3}) the infimum is realized on $E(\hat{\lambda}_1)$.

In (\ref{eq4}) both the supremum and the infimum are realized on $E(\hat{\lambda}_m)$.

From these properties, it is clear that the elements of $E(\hat{\lambda}_1)$ have constant sign while for $m \geq 2$ the elements of $E(\hat{\lambda}_m)$ are nodal (that is, sign changing). Let $\hat{u}_1$ denote the $L^2$-normalized (that is, $||\hat{u}_1||_2=1$) positive eigenfunction corresponding to $\hat{\lambda}_1$. As we have already mentioned, $\hat{u}_1\in C_+ \backslash \left\{ 0 \right\}$. Using Harnack's inequality (see, for example Motreanu, Motreanu and Papageorgiou \cite[p. 212]{10}), we have that $\hat{u}_1 (z)>0$ for all $z\in \Omega$. Moreover, if $\xi^+ \in L^\infty (\Omega)$, then using the strong maximum principle, we have $\hat{u}_1 \in D_+$.

The following useful inequalities are also easy consequences of the above properties.
\begin{proposition}\label{prop5}
	\begin{itemize}
		\item[(a)] If $m\in \NN,\ \eta\in L^\infty (\Omega),\ \eta(z)\leq \hat{\lambda}_m\ \mbox{for almost all}\ z\in \Omega,\ \eta \not\equiv \hat{\lambda}_m,$ then $\vartheta(u)-\int_\Omega \eta (z)u^2 dz \geq c_1 ||u||^2\ \mbox{for all}\ u\in \overline{\mathop{\oplus}\limits_{k \geq m} E(\hat{\lambda}_k)}\ \mbox{and for some}\ c_1>0$.
		\item[(b)] If $m\in \NN,\ \eta\in L^\infty (\Omega),\ \eta(z)\geq \hat{\lambda}_m\ \mbox{for almost all}\ z\in \Omega,\ \eta \not\equiv \hat{\lambda}_m,$ then $\vartheta(u)-\int_\Omega \eta (z)u^2 dz \leq - c_2 ||u||^2\ \mbox{for all}\ u\in \mathrel{\mathop{\oplus}^{m}_{k=1}} E(\hat{\lambda}_k)\ \mbox{and for some}\ c_2>0$.
	\end{itemize}
\end{proposition}

Finally, let us fix some basic notations and terminology. So, by\\ $A\in\mathcal{L}(H^1(\Omega), H^1(\Omega)^*)$ we denote the linear operator defined by
$$\left\langle A(u),h\right\rangle = \int_\Omega (Du,Dh)_{\RR^N}dz\ \mbox{for all}\ u,h\in H^1(\Omega).$$

A Banach space $X$ is said to have the ``Kadec-Klee property" if the following holds
$$``u_n \stackrel{w}{\rightarrow}u\ \mbox{in}\ X\ \mbox{and}\ ||u_n||\rightarrow||u||\Rightarrow u_n\rightarrow u\ \mbox{in}\ X".$$

Locally uniformly convex Banach spaces, in particular Hilbert spaces, have the Kadec-Klee property.

Let $x\in\RR$. We set $x^\pm =\max\left\{ \pm x,0 \right\}$ and  for $u\in H^1(\Omega)$ we define
$$u^\pm(\cdot)=u(\cdot)^\pm.$$

We know that
$$u^\pm \in H^1(\Omega),\ |u|=u^+ + u^-,\ u=u^+-u^-.$$

By $|\cdot|_N$ we denote the Lebesgue measure on $\RR^N$. Also, if $\varphi\in C^1 (X,\RR)$ then
$$K_\varphi= \left\{ u\in X: \varphi'(u)=0\right\}\ \mbox{(the critical set of}\ \varphi ).$$

If $p\in\left[1,\infty\right)$, then $p'\in(1,+\infty]$ and $\frac{1}{p}+\frac{1}{p'}=1$. Finally, we set
$$n_0=\max\left\{ k \in \NN :\hat{\lambda}_k \leq 0 \right\}.$$

If $\hat{\lambda}_k>0$ for all $k\in \NN$ (this is the case if $\xi \geq 0$ and $\xi \not\equiv 0 $ or $\beta\not\equiv 0$), then we set $n_0=0$.

\section{Positive solutions}

The hypotheses on the data of problem (\ref{eqp}) are the following:

\smallskip
$H(\xi):\ \xi\in L^s(\Omega)\ \mbox{with}\ s>N\ \mbox{and}\ \xi^+ \in L^\infty (\Omega)$.

\smallskip
$H(\beta):\ \beta\in W^{1,\infty}(\partial\Omega)\ \mbox{and}\ \beta(z)\geq 0\ \mbox{for all}\ z\in \partial\Omega$.

\smallskip
$H(g):\ g:\Omega\times\RR\rightarrow\RR\ \mbox{is a Carath\'eodory function such that}$
\begin{itemize}
	\item[(i)] for every $\rho>0$, there exists $a_\rho\in L^\infty(\Omega)_+$ such that
	$$g(z,x)\leq a_\rho (z)\ \mbox{for almost all}\ z\in \Omega\ \mbox{and all}\ x\in [0,\rho];$$
	\item[(ii)] $\lim\limits_{x\rightarrow +\infty}\frac{g(z,x)}{x} = 0$ uniformly for almost all $z\in\Omega$;
	\item[(iii)] there exist constants $0<c_3<c_4$ and $q\in (1,2)$ such that
	$$c_3 x^{q-1}\leq g(z,x)\ \mbox{for almost all}\ z\in \Omega\ \mbox{and all}\ x\geq 0,$$
	$$\limsup\limits_{x\rightarrow 0^+}\frac{g(z,x)}{x^{q-1}}\leq c_4\ \mbox{uniformly for almost all}\ z\in\Omega ;$$
	\item[(iv)] if $G(z,x)=\int^x_0 g(z,s)ds$,\ \mbox{then}\ $g(z,x)x-2G(z,x)\leq 0$ for almost all $z\in\Omega$ and all $x\geq 0$;
	\item[(v)] for every $\rho > 0$, there exists $\hat{\xi}_\rho > 0$ such that for almost all $z\in \Omega$ the function
	$$x\mapsto g(z,x)+\hat{\xi}_\rho x$$
	is nondecreasing on $[0,\rho]$.
\end{itemize}

\smallskip
$H(f):\ f:\Omega\times\RR\rightarrow\RR$ is a Carath\'eodory function such that
\begin{itemize}
	\item[(i)] $|f(z,x)|\leq a(z)(1+x^{r-1})$ for almost all $z\in\Omega$ and all $x\geq 0$ with $a\in L^\infty(\Omega),\ 2<r<2^*$;
	\item[(ii)] $\lim\limits_{x\rightarrow +\infty}\frac{f(z,x)}{x}=+\infty$ uniformly for almost all $z\in\Omega$;
	\item[(iii)] $\lim\limits_{x\rightarrow 0}\frac{f(z,x)}{x}=0$ uniformly for almost all $z\in\Omega$ and there exists $\delta_0 >0$ such that
	$$f(z,x)\geq 0\ \mbox{for almost all}\ z\in\Omega\ \mbox{and all}\ x\in[0,\delta_0];$$
	\item[(iv)] for every $\rho > 0$, there exists $\tilde{\xi}_\rho >0$ such that for almost all $z\in\Omega$ the function
	$$x\rightarrow f(z,x)+\tilde{\xi}_\rho x$$
	is nondecreasing on $[0,\rho]$
\end{itemize}

We set $F(z,x)=\int^x_0 f(z,s)ds$ and define
$$\gamma_\lambda(z,x)=\lambda g(z,x)+f(z,x)-2[\lambda G(z,x)+F(z,x)]\ \mbox{for all}\ (z,x)\in \Omega\times\RR_+.$$

\smallskip
$H_0:$ For every $\lambda > 0$, there exists $e_\lambda\in L^1(\Omega)$ such that
$$\gamma_\lambda (z,x)\leq\gamma_\lambda(z,y)+e_\lambda(z)\ \mbox{for almost all}\ z\in\Omega\ \mbox{and all}\ 0\leq x \leq y.$$
\begin{remark}
Since we are looking for positive solutions and all of the above hypotheses concern the positive semi-axis $\RR_{+}=[0,+\infty)$,  we may assume without any loss of generality that
$$g(z,x)=f(z,x)=0\ \mbox{for almost all}\ z\in\Omega\ \mbox{all}\ x\leq 0$$
(note that hypotheses $H(g)(iii)$ and $H(f)(iii)$ imply that $g(z,0)=f(z,0)=0$ for almost all $z\in\Omega$). Hypothesis $H(g)(ii)$ implies that for almost all $z\in\Omega,\ g(z,\cdot)$ is
strictly sublinear near $+\infty$. This, together with hypothesis $H(g)(iii)$, implies that
$g(z,\cdot)$ is
globally the ``concave" contribution to the reaction of problem (\ref{eqp}). On the other hand, hypothesis $H(f)(ii)$ implies that for almost all $z\in\Omega,\ f(z,\cdot)$ is strictly superlinear near $+\infty$. Hence $f(z,x)$ is globally the ``convex" contribution to the reaction of (\ref{eqp}). Therefore on the right-hand side (reaction) of problem (\ref{eqp}), we have the competition of concave and convex nonlinearities (``concave-convex problem"). We stress that the superlinearity of $f(z,\cdot)$ is not expressed using the well-known Ambrosetti-Rabinowitz condition (see Ambrosetti and Rabinowitz \cite{3}). Instead, we use hypothesis $H_0$, which is a slightly more general version of a condition used by Li and Yang \cite{9}. Hypothesis $H_0$ is less restrictive than the Ambrosetti-Rabinowitz superlinearity condition and permits the consideration of superlinear terms with ``slower" growth near $+\infty$, which fail to satisfy the AR-condition (see the examples below). Hypothesis $H_0$ is a quasimonotonicity condition on $\gamma_\lambda(z,\cdot)$ and it is satisfied if there exists $M=M(\lambda)>0$ such that for almost all $z\in\RR$,
$$x\mapsto\frac{\lambda g(z,x)+f(z,x)}{x}$$
is nondecreasing on $[M,+\infty)$ (see \cite{9}).
\end{remark}

{\bf Examples}. The following pair satisfies hypotheses $H(g)$ and $H(f)$:
$$g(z,x)=a(z)x^{q-1},\ f(z,x)=b(z)x^{r-1}\ \mbox{for all}\ x\geq 0$$
with $a,b\in L^\infty(\Omega),\ a(z),b(z)>0$ for almost all $z\in\Omega$ and $1<q<2<r<2^*$. If $a\equiv b \equiv 1$, this is the reaction pair used by Ambrosetti, Brezis and Cerami \cite{2} in the context of Dirichlet problems with zero potential (that is, $\xi\equiv 0$). The above reaction pair was used by R\u adulescu and Repov\v s \cite{15}, again for Dirichlet problems with $\xi\equiv 0$.

Another possibility of a reaction pair which satisfies hypotheses $H(g)$ and $H(f)$ are the following functions (for the sake of simplicity, we drop the $z$-dependence):
\begin{eqnarray*}
&&g(x) = \left\{\begin{array}{ll}
	2x^{q-1}-x^{\tau-1}	&	\mbox{if}\ 0\leq x\leq 1\\
	x^{\eta-1}			&	\mbox{if}\ 1<x
	\end{array}\right. \mbox{with}\ 1<q,\eta<2<\tau\\
	&&\hspace{3cm}\mbox{and}\ f(x)=x\ln(1+x)\ \mbox{for all}\ x\geq 0.
\end{eqnarray*}

In this pair, the superlinear term $f(x)$ fails to satisfy the Ambrosetti-Rabinowitz condition.

Let $\mu > 0$ be as in (\ref{eq2}) and $\lambda > 0$. Let $k_\lambda : \Omega \times \RR\rightarrow\RR$ be the Carath\'eodory function defined by
\begin{equation}\label{eq5}
	k_\lambda(z,x)=\lambda g(z,x)+f(z,x)+\mu x^+\,.
\end{equation}

We set $K_\lambda(z,x)=\int^x_0 k_\lambda(z,s)ds$ and consider the $C^1$-functional $\hat{\varphi}_\lambda:H^1(\Omega)\rightarrow\RR$ defined by
$$\hat{\varphi}_\lambda(u)=\frac{1}{2}\vartheta(u)+\frac{\mu}{2}||u||^2_2-\int_\Omega K_\lambda (z,u)dz\ \mbox{for all}\ u\in H^1(\Omega).$$
\begin{proposition}\label{prop6}
	If hypotheses $H(\xi),\ H(\beta),\ H(g),\ H(f)$ and $H_0$ hold, then for every $\lambda > 0$ the functional $\hat\varphi_\lambda$ satisfies the C-condition.	
\end{proposition}
\begin{proof}
	Let $\{u_n\}_{n\geq 1}\subseteq H^1(\Omega)$ be a sequence such that
	\begin{eqnarray}
		&&|\hat{\varphi}_\lambda (u_n) |\leq M_1\ \mbox{for some}\ M_1>0\ \mbox{and all}\ n\in\NN,\label{eq6}\\
		&&(1+||u_n||)\hat{\varphi}'_\lambda(u_n)\rightarrow 0\ \mbox{in}\ H^1(\Omega)^*\ \mbox{as}\ n\rightarrow \infty.\label{eq7}
	\end{eqnarray}
	
	By (\ref{eq7}) we have
	\begin{equation}\label{eq8}
		\left|	\left\langle A(u_n),h \right\rangle \!+\! \int_\Omega (\xi(z)\!+\!\mu)u_n hdz \!+\! \int_{\partial\Omega}\beta(z)u_n hd\sigma \!-\! \int_\Omega k_\lambda (z,u_n)hdz \right| \leq \frac{\epsilon_n ||h||}{1+||u_n||}\,,
	\end{equation}
for all $h\in H^1(\Omega)$ with $\epsilon_n\rightarrow 0^+$.
	
	In (\ref{eq8}) we choose $h=-u^-_n\in H^1(\Omega)$. Then
	\begin{eqnarray}\label{eq9}
		&				&\vartheta (u^-_n)+\mu ||u^-_n||^2_2 \leq \epsilon_n\ \mbox{for all}\ n\in\NN\ \mbox{(see (\ref{eq5})),}\nonumber \\
		& \Rightarrow	& c_0||u^-_n||^2 \leq \epsilon_n\ \mbox{for all}\ n\in\NN\ \mbox{(see (\ref{eq2})),}\nonumber \\
		& \Rightarrow	& u^-_n\rightarrow 0\ \mbox{in}\ H^1(\Omega).
	\end{eqnarray}
	
	It follows from (\ref{eq6}) and (\ref{eq9}) that
	\begin{equation}\label{eq10}
		\vartheta(u^+_n)-\int_\Omega 2\left[ \lambda G(z,u^+_n)+F(z,u^+_n) \right]dz\leq M_2\ \mbox{for some}\ M_2>0\ \mbox{and all}\ n\in\NN.
	\end{equation}
	
	If in (\ref{eq8}) we choose $h=u^+_n\in H^1(\Omega)$, then
	\begin{equation}\label{eq11}
		-\vartheta(u^+_n)+\int_\Omega \left[ \lambda g(z,u^+_n)+f(z,u^+_n) \right]u^+_n dz\leq \epsilon_n\ \mbox{for all}\ n\in\NN.
	\end{equation}
	
	Adding (\ref{eq10}) and (\ref{eq11}), we obtain
	\begin{equation}\label{eq12}
		\int_\Omega \gamma_{\lambda} (z,u^+_n) dz\leq M_3\ \mbox{for some}\ M_3 > 0\ \mbox{and all}\ n\in\NN.
	\end{equation}

{\bf Claim.} {\it
		$\left\{u^+_n\right\}_{n\geq 1}\subseteq H^1(\Omega)$ is bounded.}
	
	We argue by contradiction. So, suppose that the claim is not true. By passing to a subsequence if necessary, we may assume that $||u^+_n||\rightarrow \infty$. Let $y_n=\frac{u^+_n}{||u^+_n||}$, $n\in\NN$. Then
	$$||y_n||=1,\ y_n \geq 0\ \mbox{for all}\ n\in\NN$$
	and so we may assume that
	\begin{equation}\label{eq13}
		y_n \stackrel{w}{\rightarrow} y\ \mbox{in}\ H^1(\Omega)\ \mbox{and}\ y_n\rightarrow y\ \mbox{in}\ L^{2s'}(\Omega)\ \mbox{and in}\ L^2(\partial\Omega), y \geq 0.
	\end{equation}
	
	Suppose that $y\neq 0$ and let  $\Omega^*=\left\{ z\in \Omega : y(z)>0 \right\}$. Then $|\Omega^*|_N>0$ and
	$$u^+_n(z)\rightarrow+\infty\ \mbox{for almost all}\ z\in \Omega^*.$$
	
	We have
	\begin{eqnarray}
		&&\frac{G(z,u^+_n)}{||u^+_n||^2}=\frac{G(z,u^+_n)}{(u^+_n)^2}y^2_n\rightarrow 0\ \mbox{for a.a.}\ z\in\Omega^*\ \mbox{(see hypothesis H(g)(ii))},\label{eq14}\\
		&&\frac{F(z,u^+_n)}{||u^+_n||^2}\!=\!\frac{F(z,u^+_n)}{(u^+_n)^2}y^2_n\rightarrow \!+\!\infty\ \mbox{for a.a.}\ z\in\Omega^*\ \mbox{(see hypothesis H(f)(ii))}.\label{eq15}
	\end{eqnarray}
	
	It follows from (\ref{eq14}), (\ref{eq15}) and Fatou's lemma that
	\begin{equation}\label{eq16}
		\lim\limits_{n\rightarrow\infty}\left[ \lambda \int_\Omega \frac{G(z,u^+_n)}{||u^+_n||^2} dz + \int_\Omega \frac{F(z,u^+_n)}{||u^+_n||^2}dz \right] = +\infty\,.
	\end{equation}
	
	On the other hand, (\ref{eq6}) and (\ref{eq9}) imply that
	\begin{equation}\label{eq17}
		 \lambda\int_\Omega \frac{G(z,u^+_n)}{||u^+_n||^2} dz + \int_\Omega \frac{F(z,u^+_n)}{||u^+_n||^2}dz \leq \frac{M_1}{||u^+_n||^2}+\frac{1}{2}\vartheta(y_n)+\frac{\mu}{2}||y_n||^2_2\leq M_4
	\end{equation}
	for some $M_4$, all $n\in \NN$.
	
	Comparing (\ref{eq16}) and (\ref{eq17}) we obtain a contradiction.
	
	Next, suppose that $y\equiv 0$. For $\eta>0$ we set $\hat{y}_n=(2\eta)^\frac{1}{2} y_n,\ n\in\NN$. Then $\hat{y}_n\rightarrow 0$ in $H^1(\Omega)$ and so we have
	\begin{equation}\label{eq18}
		\int_\Omega G(z,\hat{y}_n)dz\rightarrow 0\ \mbox{and}\ \int_\Omega F(z,\hat{y}_n)dz\rightarrow 0.
	\end{equation}
	
	Since $||u^+_n||\rightarrow \infty$, we can find $n_1\in\NN$ such that
	\begin{equation}\label{eq19}
		\frac{(2\eta)^{\frac{1}{2}}}{||u^+_n||}\in (0,1]\ \mbox{for all}\ n\geq n_1.
	\end{equation}
	
	We choose $t_n\in [0,1]$ such that
	\begin{equation}\label{eq20}
		\hat{\varphi_\lambda}(t_n u^+_n)=\max\left[ \hat{\varphi}_\lambda(t u^+_n):0\leq t\leq 1 \right].
	\end{equation}
	
	From (\ref{eq19}), (\ref{eq20}) we have
	\begin{align}\label{eq21}
	\hat{\varphi}_\lambda(t_n u^+_n) \geq 	& \hat{\varphi}_\lambda (\hat{y}_n)\ \mbox{(see (\ref{eq19}))} \nonumber \\
									 =		& \frac{1}{2}\vartheta(\hat{y}_n)-\lambda\int_\Omega G(z,\hat{y}_n)dz - \int_\Omega F(z,\hat{y}_n)dz\ \mbox{(see (\ref{eq5}))} \nonumber \\
									 \geq	& \eta - \lambda\int_\Omega G(z,\hat{y}_n)dz-\int_\Omega F(z,\hat{y}_n)dz \nonumber \\
									 \geq	& \frac{1}{2}\eta\ \mbox{for all}\ n\geq n_2 \geq n_1\ \mbox{(see (\ref{eq18}))}.
	\end{align}
	
	Since $\eta >0$ is arbitrary,  we infer from (\ref{eq21}) that
	\begin{equation}\label{eq22}
		\hat{\varphi}_\lambda(t_n u^+_n)\rightarrow+\infty\ \mbox{as}\ n\rightarrow\infty\,.
	\end{equation}
	
	We know that
	\begin{align*}
		             & \hat{\varphi}_\lambda(0)=0\ \mbox{and}\ \hat{\varphi}_\lambda(u^+_n)\leq M_5\ \mbox{for some}\ M_5 >0\ \mbox{and all}\ n\in\NN\ \mbox{(see (\ref{eq6}) and (\ref{eq9}))}, \\
		 \Rightarrow & t_n\in (0,1)\ \mbox{for all}\ n\geq n_3\ \mbox{(see (\ref{eq22}))}.
	\end{align*}
	
	So, (\ref{eq20}) implies that
	\begin{eqnarray}\label{eq23}
		&             & t_n\frac{d}{dt}\hat{\varphi}_\lambda (t u^+_n)|_{t=t_n}=0 ,\nonumber \\
		& \Rightarrow & \left\langle \hat{\varphi}'_{\lambda} (t_n u^+_n),t_n u^+_n \right\rangle = 0\ \mbox{(by the chain rule)},\nonumber \\
		& \Rightarrow & \vartheta(t_n u^+_n)=\int_\Omega \left[ \lambda g(z,t_n u^+_n) + f(z,t_n u^+_n) \right](t_n u^+_n)dz\ \mbox{for all}\ n\geq n_3 .
	\end{eqnarray}
	
	We have $0 \leq t_n u^+_n\leq u^+_n $. Then hypothesis $H_0$ implies that
	\begin{align}\label{eq24}
		            & \gamma_\lambda(z,t_n u^+_n)\leq\gamma_\lambda(z,u^+_n)+e_\lambda(z)\ \mbox{for almost all}\ z\in\Omega\ \mbox{and all}\ n\geq n_3 , ,\nonumber \\
		  \Rightarrow & \int_\Omega \gamma_\lambda (z,t_n u^+_n)dz \leq \int_\Omega\gamma_\lambda(z,u^+_n)dz + ||e_\lambda||_1 \leq M_6\ \mbox{for some}\ M_6 > 0,\ \mbox{all}\ n\geq n_3 \\
		 			  & (\mbox{see}\ (\ref{eq12})). \nonumber
	\end{align}
	
	We return to (\ref{eq23}), add to both sides $-2\int_\Omega [\lambda G(z,t_n u^+_n)+F(z,t_n u^+_n)]dz$ and use (\ref{eq24}).
	
	Then
	\begin{equation}\label{eq25}
		2\hat{\varphi}_\lambda(t_n u^+_n)\leq M_6\ \mbox{for all}\ n\geq n_3 .
	\end{equation}
	
	Comparing (\ref{eq22}) and (\ref{eq25}) again, we get a contradiction.
	
	This proves the claim.
	
	Then (\ref{eq9}) and the claim imply that $\left\{ u_n \right\}_{n\geq 1} \subseteq H^1 (\Omega)$ is bounded. So, we may assume that
	\begin{equation}\label{eq26}
		u_n \stackrel{w}{\rightarrow} u\ \mbox{in}\ H^1(\Omega)\ \mbox{and}\ u_n\rightarrow u\ \mbox{in}\ L^{2s'} (\Omega)\ \mbox{and in}\ L^2(\partial\Omega).
	\end{equation}
	
	In (\ref{eq8}) we choose $h=u_n -u\in H^1 (\Omega)$, pass to the limit as $n\rightarrow \infty$ and use (\ref{eq26}).
	
	Then
	\begin{eqnarray*}
		&             & \lim\limits_{n\rightarrow\infty} \left\langle A(u_n),u_n - u \right\rangle = 0, \\
		& \Rightarrow & ||Du_n||_2\rightarrow ||Du||_2 , \\
		& \Rightarrow & u_n \rightarrow u\ \mbox{in}\ H^1(\Omega)\ \mbox{(by the Kadec-Klee property, see (\ref{eq26}))} \\
		& \Rightarrow & \hat{\varphi}_\lambda\ \mbox{satisfies the C-condition.}
	\end{eqnarray*}
\end{proof}

Let $\mathcal{L}=\{\lambda>0:\mbox{problem \eqref{eqp} admits a positive solution}\}$.
$$S^{\lambda}_{+}=\mbox{set of positive solutions of \eqref{eqp}}.$$
\begin{proposition}\label{prop7}
	If hypotheses $H(\xi),H(\beta),H(g),H(f)$ and $H_0$ hold, then $\mathcal{L}\neq\emptyset$ and when $\lambda\in\mathcal{L}$ we have $\left(0,\lambda\right]\subseteq\mathcal{L}$ and for all $\lambda>0$, $S^{\lambda}_{+}\subseteq D_+.$
\end{proposition}
\begin{proof}
	Hypotheses $H(g)(i)\rightarrow(iii)$ imply that given $\epsilon>0$, we can find $c_5=c_5(\epsilon)>0$ such that
	\begin{equation}\label{eq27}
		G(z,x)\leq\frac{\epsilon}{2}x^2+c_5x^q\ \mbox{for almost all}\ z\in\Omega\ \mbox{and all}\ x\geq 0.
	\end{equation}
	
	Similarly, hypotheses $H(f)(i),(iii)$ imply that given $\epsilon>0$, we can find $c_6=c_6(\epsilon)>0$ such that
	\begin{equation}\label{eq28}
		F(z,x)\leq\frac{\epsilon}{2}x^2+c_6x^r\ \mbox{for almost all}\ z\in\Omega\ \mbox{and all}\ x\geq 0.
	\end{equation}
	
	We set
	$$\bar{H}_{n_0}=\overset{n_0}{\underset{\mathrm{k=1}}\oplus}E(\hat{\lambda}_k)\ \mbox{and}\ \hat{H}_{n_0+1}=\bar{H}^\perp_{n_0}=\overline{{\underset{\mathrm{k\geq n_0+1}}\oplus}E(\hat{\lambda}_k)}.$$
	
	We have
	$$H^1(\Omega)=\bar{H}_{n_0}\oplus\hat{H}_{n_0+1}.$$
	
	Recall that $n_0=\max\{k\in\NN:\hat{\lambda}_k\leq 0\}$. We set $n_0=0$ if $\hat{\lambda}_k>0$ for all $k\in\NN$ and this is the case if $\xi\geq 0$ and $\xi\not\equiv 0$ or $\beta\not\equiv 0$. Then $\bar{H}_{n_0}=\{0\}$ and $\hat{H}_{n_0+1}=H^1(\Omega)$.
	
	Let $u\in\hat{H}_{n_0+1}$. Then
	\begin{align}\label{eq29}
		 \hat{\varphi}_{\lambda}(u)=&\frac{1}{2}\vartheta(u)+\frac{\mu}{2}||u||^2_2-\int_{\Omega}K_{\lambda}(z,u)dz\nonumber\\
		\geq&\frac{1}{2}\vartheta(u)-\lambda\int_{\Omega}G(z,u)dz-\int_{\Omega}F(z,u)dz\ (\mbox{see (\ref{eq5})})\nonumber\\
		 \geq&\frac{1}{2}\vartheta(u)-\frac{\epsilon}{2}(\lambda+1)||u||^2_2-c_7(\lambda||u||^q+||u||^r)\ \mbox{for some}\ c_7>0,
	\end{align}
see (\ref{eq27}) and (\ref{eq28}).
	
	Since $u\in\hat{H}_{n_0+1}$, from Proposition \ref{prop5}(a) and by choosing $\epsilon\in(0,1)$ small enough, we have
	\begin{equation}\label{eq30}
		\frac{1}{2}[\vartheta(u)-\epsilon(\lambda+1)||u||^2_2]\geq c_8||u||^2\ \mbox{for some}\ c_8>0.
	\end{equation}
	
	We use (\ref{eq30}) in (\ref{eq29}). Then
	\begin{equation}\label{eq31}
		 \hat{\varphi}_{\lambda}(u)\geq\left[c_8-c_7(\lambda||u||^{q-2}+||u||^{r-2})\right]||u||^2.
	\end{equation}
	
	Let $\Im_{\lambda}(t)=\lambda t^{q-2}+t^{r-2},\ t>0$. We have $1<q<2<r$. Hence
	$$\Im_{\lambda}(t)\rightarrow+\infty\ \mbox{as}\ t\rightarrow 0^+\ \mbox{and as}\ t\rightarrow+\infty.$$
	
	So, we can find $t_0>0$ such that
	\begin{eqnarray}\label{eq32}
	 \Im_{\lambda}(t_0)=\min\limits_{\RR_+}\Im,
		&\Rightarrow&\Im'_{\lambda}(t_0)=0,\nonumber\\
		&\Rightarrow&\lambda(2-q)=(r-2)t^{r-q}_0,\nonumber\\
		&\Rightarrow&t_0=\left[\frac{\lambda(2-q)}{r-2}\right]^{\frac{1}{r-q}}.
	\end{eqnarray}
	
	Then we have
	\begin{eqnarray*}
		 &&\Im_{\lambda}(t_0)=\lambda\frac{(r-2)^{\frac{2-q}{r-q}}}{(\lambda(2-q))^{\frac{2-q}{r-q}}}+\left(\frac{\lambda(2-q)}{r-2}\right)^{\frac{r-2}{r-q}},\\
		&\Rightarrow&\Im_{\lambda}(t_0)\rightarrow 0\ \mbox{as}\ \lambda\rightarrow 0^+\ (\mbox{since}\ \frac{2-q}{r-q}<1).
	\end{eqnarray*}
	
	Therefore, we can find $\lambda_0>0$ such that
	$$\Im_{\lambda}(t_0)<\frac{c_8}{c_7}\ \mbox{for all}\ \lambda\in(0,\lambda_0).$$
	
	Returning to (\ref{eq31}), we deduce that there exists a positive number $m_\lambda$ such that
	\begin{equation}\label{eq33}
		\hat{\varphi}_{\lambda}(u)\geq m_{\lambda}>0=\hat{\varphi}_{\lambda}(0)\ \mbox{for all}\ u\in\hat{H}_{n_0+1}\ \mbox{with}\ ||u||=\hat{\rho}_{\lambda}=t_0(\lambda).
	\end{equation}
	
	On the other hand, hypotheses $H(f)(i),(ii)$ imply that if $\tau>\mu+\hat{\lambda}_{n_0+1}$, then we can find $c_9=c_9(\tau)>0$ such that
	\begin{equation}\label{eq34}
		F(z,x)\geq\frac{\tau}{2}x^2-c_9x^r\ \mbox{for almost all}\ z\in\Omega\ \mbox{and all}\ x\geq 0.
	\end{equation}
	
	Let $w_0\in E(\hat{\lambda}_{n_0+1})$ with $||w_0||=1$. We consider the space
	$$Y=\bar{H}_{n_0}\oplus\RR w_0.$$
	
	This is a finite dimensional subspace of $H^1(\Omega)$ and if $u\in Y$, then we can write $u$ in a unique way as
	$$u=\bar{u}+\alpha w_0\ \mbox{with}\ \bar{u}\in\bar{H}_{n_0}\ \mbox{and}\ \alpha\in\RR.$$
	
	Exploiting the orthogonality of the component spaces and since $G\geq 0$ (see hypothesis $H(g)(iii)$), we have
	\begin{eqnarray}\label{eq35} &&\hat{\varphi}_{\lambda}(u)\leq\frac{1}{2}\vartheta(\bar{u})+\frac{\alpha^2}{2}\vartheta(w_0)+\frac{\mu}{2}||\bar{u}||^2_2+\frac{\mu}{2}\alpha^2||w_0||^2_2\nonumber\\
		 &&\hspace{2cm}-\frac{\tau}{2}||\bar{u}||^2_2-\frac{\tau}{2}\alpha^2||w_0||^2_2+c_9||\bar{u}+\alpha w_0||^r_r\ \mbox{(see (\ref{eq34}))}.
	\end{eqnarray}
	
	Note that
	\begin{eqnarray}
		&&\vartheta(\bar{u})\leq\hat{\lambda}_{n_0}||\bar{u}||^2_2\leq 0\ (\mbox{see (\ref{eq4}) and recall that}\ \hat{\lambda}_{n_0}\leq 0),\label{eq36}\\
		&&\vartheta(w_0)=\hat{\lambda}_{n_0+1}||w_0||^2_2\ (\mbox{since}\ w_0\in E(\hat{\lambda}_{n_0+1})).\label{eq37}
	\end{eqnarray}
	
	Returning to (\ref{eq35}) and using (\ref{eq36}), (\ref{eq37}), we obtain
	\begin{align}\label{eq38}
		 \hat{\varphi}_{\lambda}(u)\leq&-\frac{\tau-\mu}{2}||\bar{u}||^2_2-\frac{\alpha^2}{2}\left[\tau-\mu-\hat{\lambda}_{n_0+1}\right]||w_0||^2_2+c_9||\bar{u}+\alpha w_0||^r_r\nonumber\\
		\leq&-c_{10}[||\bar{u}||^2_2+\alpha^2||w_0||^2_2]+c_9||\bar{u}+\alpha w_0||^r_r\ \mbox{with}\ c_8=\tau-\mu-\hat{\lambda}_{n_0+1}>0\nonumber\\
		=&-c_{10}||\bar{u}+\alpha w_0||^2_2+c_9||\bar{u}+\alpha w_0||^r_r\\
		&\mbox{(by the orthogonality of the component spaces)}.\nonumber
	\end{align}
	
	Since $Y$ is finite dimensional, all norms are equivalent. So, by (\ref{eq38}) we have
	\begin{equation}\label{eq39}
		\hat{\varphi}_{\lambda}(u)\leq c_{11}||\bar{u}+\alpha w_0||^r-c_{12}||\bar{u}+\alpha w_0||^2\ \mbox{for some}\ c_{11},c_{12}>0.
	\end{equation}
	
	But $r>2$. So,  it follows from (\ref{eq39}) that we can find $\rho\in(0,1)$ small such that
	\begin{equation}\label{eq40}
		\hat{\varphi}_{\lambda}(u)\leq 0\ \mbox{for all}\ u\in Y\ \mbox{with}\ ||u||\leq\rho.
	\end{equation}
	
	By (\ref{eq32}) and (\ref{eq33}) we see that
	$$\hat{\rho}_{\lambda}\rightarrow 0^+\ \mbox{as}\ \lambda\rightarrow 0^+$$
	
	Therefore we can find $\lambda_1\leq\lambda_0$ such that
	$$\hat{\rho}_{\lambda}\in(0,\rho)\ \mbox{for all}\ \lambda\in(0,\lambda_1).$$
	
	We consider the following sets:
	\begin{eqnarray*}
		&&E_0=\{u=\bar{u}+\alpha w_0:\bar{u}\in\bar{H}_{n_0},\alpha\in[0,1]\ \mbox{and}\ (||u||=\rho,\alpha\in[0,1])\ \mbox{or}\\
		&&\hspace{1cm}(||u||\leq\rho,\alpha\in\{0,1\})\},\\
		&&E=\{u=\bar{u}+\alpha w_0:\bar{u}\in\bar{H}_{n_0},||u||\leq\rho,\alpha\in[0,1]\},\\
		&&D=\hat{H}_{n_0+1}\cap\partial B_{\rho_{\lambda}}.
	\end{eqnarray*}
	
	From Gasinski and Papageorgiou \cite[p. 643]{6}  we know that
	\begin{equation}\label{eq41}
		\{E_0,E\}\ \mbox{links with}\ D\ \mbox{in}\ H^1(\Omega)
	\end{equation}
(see Definition \ref{def1} and recall that $1>\rho>\hat{\rho}_{\lambda}$).

By Proposition \ref{prop6} we know that for all $\lambda>0$
\begin{equation}\label{eq42}
	\hat{\varphi}_{\lambda}\ \mbox{satisfies the C-condition}.
\end{equation}

On account of (\ref{eq33}), (\ref{eq40}), (\ref{eq41}), (\ref{eq42}), we can apply Definition \ref{def1} (the linking theorem) and find $u\in H^1(\Omega)$ such that
\begin{equation}\label{eq43}
	u_0\in K_{\hat{\varphi}_{\lambda}},\ \hat{\varphi}_{\lambda}(0)=0<m_{\lambda}<\hat{\varphi}_{\lambda}(u_0),
\end{equation}
where $m_\lambda$ is the same as in relation \eqref{eq33}.

It follows from (\ref{eq43}) that $u_0\neq 0$ and
	\begin{equation}\label{eq44}
		\left\langle A(u_0),h\right\rangle+\int_{\Omega}(\xi(z)+\mu)u_0hdz+\int_{\partial\Omega}\beta(z)u_0hd\sigma=\int_{\Omega}k_{\lambda}(z,u_0)hdz\ \mbox{for all}\ h\in H^1(\Omega).
	\end{equation}
	
	In (\ref{eq44}) we choose $h=-u^-_0\in H^1(\Omega)$. Then
	\begin{eqnarray*}
		 \vartheta(u^-_0)+\mu||u^-_0||^2_2=0\ (\mbox{see (\ref{eq5})}),
		&\Rightarrow&c_0||u^-_0||^2\leq 0\ (\mbox{see (\ref{eq2})}),\\
		&\Rightarrow&u_0\geq 0,\ u_0\neq 0.
	\end{eqnarray*}
	
	So, equation (\ref{eq44}) becomes
	\begin{eqnarray}\label{eq45}
		&&\left\langle A(u_0),h\right\rangle+\int_{\Omega}\xi(z)u_0hdz+\int_{\partial\Omega}\beta(z)u_0hd\sigma=\int_{\Omega}[\lambda g(z,u_0)+f(z,u_0)]hdz\nonumber\\
		&&\mbox{for all}\ h\in H^1(\Omega),\nonumber\\
		&\Rightarrow&-\Delta u_0(z)+\xi(z)u_0(z)=\lambda g(z,u_0(z))+f(z,u_0(z))\ \mbox{for almost all}\ z\in\Omega,\nonumber\\
		&&\frac{\partial u_0}{\partial n}+\beta(z)u=0\ \mbox{on}\ \partial\Omega
	\end{eqnarray}
	(see Papageorgiou and R\u adulescu \cite{11}).
	
	We set
	$$e_{\lambda}(z,x)=\lambda g(z,x)+f(z,x)\ \mbox{and}\ \hat{a}_{\lambda}(z)=\frac{e_{\lambda}(z,u_0(z))}{1+u_0(z)}\,.$$
	
	Hypotheses $H(g)(i),(ii)$ and $H(f)(i)$ imply that
	\begin{equation}\label{eq46}
		|e_{\lambda}(z,x)|\leq c_{13}(1+x^{r-1})\ \mbox{for almost all}\ z\in\Omega,\ \mbox{all}\ x\geq 0,\ \mbox{some}\ c_{13}=c_{13}(\lambda)>0.
	\end{equation}
	
	Then we have
	\begin{align*}
		|\hat{a}_{\lambda}(z)|=&\frac{|e_{\lambda}(z,u_0(z))|}{1+u_0(z)}\\
		\leq&\frac{c_{13}(1+u_0(z)^{r-1})}{1+u_0(z)}\ (\mbox{see (\ref{eq46})})\\
		\leq&\frac{c_{14}(1+u_0(z))^{r-1}}{1+u_0(z)}\ \mbox{for some}\ c_{14}=c_{14}(\lambda)>0\\
		=&c_{14}(1+u_0(z))^{r-2}\ \mbox{for almost all}\ z\in\Omega,\\
		\Rightarrow\hat{a}_{\lambda}\in&L^{\tau}(\Omega)\ \mbox{with}\ \tau>\frac{N}{2}
	\end{align*}
	(note that $(r-2)\frac{N}{2}<\left(\frac{2N}{N-2}-2\right)\frac{N}{2}=\frac{2N}{N-2}=2^*$ if $N\geq 3$).
	
	We rewrite (\ref{eq45}) as
	\begin{eqnarray*}
		&&-\Delta u_0(z)+\xi(z)u_0(z)=\hat{a}_{\lambda}(z)(1+u_0(z))\ \mbox{for almost all}\ z\in\Omega,\\
		&&\frac{\partial u_0}{\partial n}+\beta(z)u_0=0\ \mbox{on}\ \partial\Omega\,.
	\end{eqnarray*}
	
	Using Lemma 5.1 of Wang \cite{16} we have
	$$u_0\in L^{\infty}(\Omega)\ (\mbox{see hypothesis}\ H(\xi)).$$
	
	Then the Calderon-Zygmund estimates (see Wang \cite[Lemma 5.2]{16}) imply that
	$$u_0\in C_+\backslash\{0\}.$$
	
	Let $\rho=||u_0||_{\infty}$. On account of hypotheses $H(g)(v)$ and $H(f)(iv)$, we can find $\bar{\xi}_{\rho}>0$ such that for almost all $z\in\Omega$, $x\mapsto \lambda g(z,x)+f(z,x)+\bar{\xi}_{\rho}x$ is nondecreasing on $[0,\rho]$. Then from (\ref{eq45}) we have
	\begin{eqnarray*}
		&&\Delta u_0(z)\leq(\bar{\xi}_{\rho}+\xi(z))u_0(z)\ \mbox{for almost all}\ z\in\Omega,\\
		&\Rightarrow&\Delta u_0(z)\leq(\bar{\xi}_{\rho}+||\xi^+||_{\infty})u_0(z)\ \mbox{for almost all}\ z\in\Omega\ (\mbox{see hypothesis}\ H(\xi)),\\
		&\Rightarrow&u_0\in D_+\ (\mbox{by the strong maximum principle}).
	\end{eqnarray*}
	
	Therefore we have proved that for $\lambda>0$ small enough, we have
	\begin{equation}\label{eq47}
		\lambda\in\mathcal{L}\ \mbox{and for every}\ \lambda\in\mathcal{L},\ S^{\lambda}_{+}\subseteq D_+\,.
	\end{equation}
	
	Next, let $\lambda\in\mathcal{L}$ and pick $\tau\in (0,\lambda)$. Since $\lambda\in\mathcal{L}$, we can find $u_{\lambda}\in S^{\lambda}_+\subseteq D_+$ (see (\ref{eq47})). We have
	\begin{align}\label{eq48}
		-\Delta u_{\lambda}(z)+\xi(z)u_{\lambda}(z)=&\lambda g(z,u_{\lambda}(z))+f(z,u_{\lambda}(z))\nonumber\\
		\geq&\tau g(z,u_{\lambda}(z))+f(z,u_{\lambda}(z))\ \mbox{for almost all}\ z\in\Omega\\
		&(\mbox{since}\ g\geq 0,\ \mbox{see hypothesis}\ H(g)(iii)).\nonumber
	\end{align}
	
	We consider the following truncation of the Carath\'eodory map $k_{\tau}(z,\cdot)$ (see (\ref{eq5}))
	\begin{equation}\label{eq49}
		\hat{k}_{\tau}(z,x)=\left\{\begin{array}{ll}
			k_{\tau}(z,x)&\mbox{if}\ x\leq u_{\lambda}(z)\\
			k_{\tau}(z,u_{\lambda}(z))&\mbox{if}\ u_{\lambda}(z)<x.
		\end{array}\right.
	\end{equation}
	
	We set $\hat{K}_{\tau}(z,x)=\int^x_0k_{\tau}(z,s)ds$ and consider the $C^1$-functional $\hat{\psi}_{\tau}:H^1(\Omega)\rightarrow\RR$ defined by
	 $$\hat{\psi}_{\tau}(u)=\frac{1}{2}\vartheta(u)+\frac{\mu}{2}||u||^2_2-\int_{\Omega}\hat{K}_{\tau}(z,u)dz\ \mbox{for all}\ u\in H^1(\Omega).$$
	
	By (\ref{eq49}) and (\ref{eq2}) it is clear that $\hat{\psi}_{\tau}(\cdot)$ is coercive. In addition, the Sobolev embedding theorem and the compactness of the trace map, imply that $\hat{\psi}_{\tau}(\cdot)$ is sequentially weakly lower semicontinuous. So, by the Weierstrass theorem, we can find $u_{\tau}\in H^1(\Omega)$ such that
	\begin{equation}\label{eq50}
		\hat{\psi}_{\tau}(u_{\tau})=\inf[\hat{\psi}_{\tau}(u):u\in H^1(\Omega)]=\hat{m}_{\tau}.
	\end{equation}
	
	With $\delta_0>0$ as in hypothesis $H(f)(iii)$, we define
	$$\hat{\delta}_0=\min\left\{\min\limits_{\overline{\Omega}}u_{\lambda},\delta_0\right\}>0\ (\mbox{recall that}\ u_{\lambda}\in D_+).$$
	
	For $u\in D_+$, choose $t\in(0,1)$ so small that
	$$tu(z)\in(0,\hat{\delta}_0]\ \mbox{for all}\ z\in\overline{\Omega}.$$
	
	Then using hypothesis $H(f)(iii)$, we have
	\begin{align}\label{eq51}
		\hat{\psi}_{\tau}(tu)\leq&\frac{t^2}{2}\vartheta(u)-\tau\int_{\Omega}G(z,tu)dz\ (\mbox{see (\ref{eq49}) and (\ref{eq5})})\nonumber\\
		\leq&\frac{t^2}{2}\vartheta(u)-\tau\frac{c_3}{q}t^q||u||^q_q\ (\mbox{see hypothesis}\ H(g)(iii)).
	\end{align}
	
	Recall that $q<2$. Then from (\ref{eq51}) and by choosing $t\in(0,1)$ even smaller if necessary, we infer that
	\begin{eqnarray*}
		 \hat{\psi}_{\tau}(tu)<0,
		&\Rightarrow&\hat{\psi}_{\tau}(u_{\tau})<0=\hat{\psi}_{\tau}(0)\ (\mbox{see (\ref{eq50})}),\\
		&\Rightarrow&u_{\tau}\neq 0.
	\end{eqnarray*}
	
	By (\ref{eq50}) we have
	\begin{align}\label{eq52}
		&\hat{\psi}'_{\tau}(u_{\tau})=0,\nonumber\\
		\Rightarrow&\left\langle A(u_{\tau}),h\right\rangle+\int_{\Omega}(\xi(z)+\mu)u_{\tau}hdz+\int_{\partial\Omega}\beta(z)u_{\tau}hd\sigma=\int_{\Omega}\hat{k}_{\tau}(z,u_{\tau})hdz\\
		&\mbox{for all}\ h\in H^1(\Omega).\nonumber
	\end{align}
	
	In (\ref{eq52}) we choose $h=-u^-_{\tau}\in H^1(\Omega)$. Then
	\begin{eqnarray*}
		 \vartheta(u^-_{\tau})+\mu||u^-_{\tau}||^2_2=0\ (\mbox{see (\ref{eq49}) and (\ref{eq5})}),
		&\Rightarrow&c_0||u^-_{\tau}||^2\leq 0\ (\mbox{see (\ref{eq2})}),\\
		&\Rightarrow&u_{\tau}\geq 0,\ u_{\tau}\neq 0.
	\end{eqnarray*}
	
	Next in (\ref{eq52}) we choose $h=(u_{\tau}-u_{\lambda})^+\in H^1(\Omega)$. Then
	\begin{align*}
		&\left\langle A(u_{\tau}),(u_{\tau}-u_{\lambda})^+\right\rangle+\int_{\Omega}(\xi(z)+\mu)u_{\tau}(u_{\tau}-u_{\lambda})^+dz+\int_{\partial\Omega}\beta(z)u_{\tau}(u_{\tau}-u_{\lambda})^+d\sigma\\
=&\int_{\Omega}[\tau g(z,u_{\lambda})+f(z,u_{\lambda})+\mu u_{\lambda}](u_{\tau}-u_{\lambda})^+dz\ (\mbox{see (\ref{eq49}) and (\ref{eq5})}),\\
		\leq&\left\langle A(u_{\lambda}),(u_{\tau}-u_{\lambda})^+\right\rangle+\int_{\Omega}(\xi(z)+\mu)u_{\lambda}(u_{\tau}-u_{\lambda})^+dz+\int_{\partial\Omega}\beta(z)u_{\lambda}(u_{\tau}-u_{\lambda})^+d\sigma
	\end{align*}
	(see (\ref{eq48}) and use Green's identity, see Gasinski and Papageorgiou \cite[p. 210]{6}),
	\begin{eqnarray*}
		 &\Rightarrow&\vartheta((u_{\tau}-u_{\lambda})^+)+\mu||(u_{\tau}-u_{\lambda})^+||^2_2\leq 0,\\
		&\Rightarrow&c_0||(u_{\tau}-u_{\lambda})^+||^2\leq 0,\\
		&\Rightarrow&u_{\tau}\leq u_{\lambda}.
	\end{eqnarray*}
	
	So, we have proved that
	\begin{eqnarray*}
		&&u_{\tau}\in[0,u_{\lambda}]=\{u\in H^1(\Omega):0\leq u(z)\leq u_{\lambda}(z)\ \mbox{for almost all}\ z\in\Omega\},\\
		&\Rightarrow&u_{\tau}\in S^{\tau}_+\ (\mbox{see (\ref{eq49})}),\\
		&\Rightarrow&\tau\in\mathcal{L}\ \mbox{and so}\ \left(0,\lambda\right]\subseteq\mathcal{L}.
	\end{eqnarray*}
\end{proof}

An interesting byproduct of the above proof is the following corollary.
\begin{corollary}\label{cor8}
	If hypotheses $H(\xi),H(\beta),H(g),H(f),H_0$ hold, $\tau,\lambda\in\mathcal{L}$ with $0<\tau<\lambda$ and $u_{\lambda}\in S^{\lambda}_+$, then we can find $u_{\tau}\in S^{\tau}_+$ such that $u_{\lambda}-u_{\tau}\in D_+$.
\end{corollary}
\begin{proof}
	An inspection of the last part of the proof of Proposition \ref{prop7} reveals that we can find $u_{\tau}\in S^{\tau}_+$ such that
	\begin{equation}\label{eq53}
		u_{\lambda}-u_{\tau}\in C_+\backslash\{0\}.
	\end{equation}
	
	Let $\rho=||u_{\lambda}||_{\infty}$ and let $\bar{\xi}_{\rho}>0$ be such that for almost all $z\in\Omega$ the function
	$$x\mapsto\lambda g(z,x)+f(z,x)+\bar{\xi}_{\rho}x$$
	is nondecreasing (see hypotheses $H(g)(v),H(f)(iv)$). We have
	\begin{align*}
		&-\Delta u_{\lambda}(z)+(\xi(z)+\bar{\xi}_{\rho})u_{\lambda}(z)\\
		=&\lambda g(z,u_{\lambda}(z))+f(z,u_{\lambda}(z))+\bar{\xi}_{\rho}u_{\lambda}(z)\\
		\geq&\lambda g(z,u_{\tau}(z))+f(z,u_{\tau}(z))+\bar{\xi}_{\rho}u_{\tau}(z)\ (\mbox{see (\ref{eq53})})\\
		\geq&\tau g(z,u_{\tau}(z))+f(z,u_{\tau}(z))+\bar{\xi}_{\rho}u_{\tau}(z)\ (\mbox{since}\ g\geq 0,\tau<\lambda)\\
		=&-\Delta u_{\tau}(z)+(\xi(z)+\bar{\xi}_{\rho})u_{\tau}(z)\ \mbox{for almost all}\ z\in\Omega\ (\mbox{recall that}\ u_{\tau}\in S^{\tau}_+)\\
		 \Rightarrow&\Delta(u_{\lambda}-u_{\tau})\leq(||\xi^+||_{\infty}+\bar{\xi}_{\rho})(u_{\lambda}-u_{\tau})(z)\ \mbox{for almost all}\ z\in\Omega\ (\mbox{see}\ H(\xi))\\
		\Rightarrow&u_{\lambda}-u_{\tau}\in D_+\ (\mbox{by the strong maximum principle}).
	\end{align*}
\end{proof}

Let $\lambda^*=\sup\mathcal{L}$.
\begin{proposition}\label{prop9}
	If hypotheses $H(\xi),H(\beta),H(g),H(f)$ and $H_0$ hold, then $\lambda^*<\infty$.
\end{proposition}
\begin{proof}
	Hypotheses $H(g)(iii)$ and $H(f)(ii),(iii)$ imply that we can find $\hat{\lambda}>0$ so big that
	\begin{equation}\label{eq54}
		\hat{\lambda}g(z,x)+f(z,x)\geq\hat{\lambda}_1x\ \mbox{for almost all}\ z\in\Omega\ \mbox{and all}\ x\geq 0.
	\end{equation}
	
	Let $\lambda>\hat{\lambda}$ and assume that $\lambda\in\mathcal{L}$. Then according to Proposition \ref{prop7} we can find $u\in S^{\lambda}_+\subseteq D_+$. Then there exists $\eta>0$ such that $\eta\hat{u}_1\leq u$. We choose the biggest such $\eta>0$. We have
	\begin{align*}
		-\Delta u(z)+\xi(z)u(z)=&\lambda g(z,u(z))+f(z,u(z))\\
		\geq&\hat{\lambda}g(z,u(z))+f(z,u(z))\ (\mbox{since}\ g\geq 0,\ \lambda>\hat{\lambda})\\
		\geq&\hat{\lambda}_1u(z)\ (\mbox{see (\ref{eq54})})\\
		\geq&\hat{\lambda}_1(\eta\hat{u}_1)(z)\\
		=&-\Delta(\eta\hat{u}_1)(z)+\xi(z)(\eta\hat{u}_1)(z)\ \mbox{for almost all}\ z\in\Omega,\\
		\Rightarrow\Delta(u-\eta\hat{u}_1)(z)\leq&||\xi^+||_{\infty}(u-\eta\hat{u}_1)(z)\ \mbox{for almost all}\ z\in\Omega\ (\mbox{see}\ H(\xi))\\
		\Rightarrow u-\eta\hat{u}_1\in D_+&(\mbox{by the strong maximum principle}).
	\end{align*}
	
	But this contradicts the maximality of $\eta>0$. So $\lambda\not\in\mathcal{L}$ and we have
	$$\lambda^*\leq\hat{\lambda}<\infty\,.$$
\end{proof}
\begin{proposition}\label{prop10}
	If hypotheses $H(\xi),H(\beta),H(g),H(f),H_0$ hold and $\lambda\in(0,\lambda^*)$, then problem \eqref{eqp} admits at least two positive solutions
	$$u_{\lambda},\hat{u}_{\lambda}\in D_+\ \mbox{and}\ \hat{u}_{\lambda}-u_{\lambda}\in D_+.$$
\end{proposition}
\begin{proof}
	Let $\nu\in(\lambda,\lambda^*)$ and let $u_{\nu}\in S^{\nu}_+\subseteq D_+$ (see Proposition \ref{prop7}). Then
	\begin{align}\label{eq55}
		-\Delta u_{\nu}(z)+\xi(z)u_{\nu}(z)=&\nu g(z,u_{\nu}(z))+f(z,u_{\nu}(z))\nonumber\\
		\geq&\lambda g(z,u_{\nu}(z))+f(z,u_{\nu}(z))\ \mbox{for almost all}\ z\in\Omega\\
		&(\mbox{since}\ g\geq 0\ \mbox{and}\ \nu>\lambda).\nonumber
	\end{align}
	
	Let $\hat{k}_{\lambda}(z,x)$ be the Carath\'eodory function defined in (\ref{eq49}), with $\tau$ replaced by $\lambda$ and $u_{\lambda}$ replaced by $u_{\nu}$. We set $\hat{K}_{\lambda}(z,x)=\int^x_0\hat{k}_{\lambda}(z,x)ds$ and consider the $C^1$-functional $\hat{\psi}_{\lambda}:H^1(\Omega)\rightarrow\RR$ defined by
	 $$\hat{\psi}_{\lambda}(u)=\frac{1}{2}\vartheta(u)+\frac{\mu}{2}||u||^2_2-\int_{\Omega}\hat{K}_{\lambda}(z,u)dz\ \mbox{for all}\ u\in H^1(\Omega).$$
	
	As in the proof of Proposition \ref{prop7}, via the Weierstrass theorem, we can find $u_{\lambda}\in H^1(\Omega)$ such that
	\begin{eqnarray*}
		 u_{\lambda}\in K_{\hat{\psi}_{\lambda}}\backslash\{0\}\subseteq[0,u_{\nu}]\cap D_+,
		&\Rightarrow&u_{\lambda}\in S^{\lambda}_+\subseteq D_+.
	\end{eqnarray*}
	
	Using this positive solution, we introduce the following truncation of $k_{\lambda}(z,\cdot)$ (see (\ref{eq5}))
	\begin{equation}\label{eq56}
		k^*_{\lambda}(z,x)=\left\{\begin{array}{ll}
			k_{\lambda}(z,u_{\lambda}(z))&\mbox{if}\ x<u_{\lambda}(z)\\
			k_{\lambda}(z,x)&\mbox{if}\ u_{\lambda}(z)\leq x.
		\end{array}\right.
	\end{equation}
	
	This is a Carath\'eodory function. We set $K^*_{\lambda}(z,x)=\int^x_0k^*_{\lambda}(z,s)ds$ and consider the $C^1$-functional $\psi^*_{\lambda}:H^1(\Omega)\rightarrow\RR$ defined by
	 $$\psi^*_{\lambda}(u)=\frac{1}{2}\vartheta(u)+\frac{\mu}{2}||u||^2_2-\int_{\Omega}K^*_{\lambda}(z,u)dz\ \mbox{for all}\ u\in H^1(\Omega).$$
	
	As before, using (\ref{eq56}) we can verify that
	\begin{equation}\label{eq57}
		K_{\psi^*_{\lambda}}\subseteq\left[u_{\lambda}\right)\cap D_+=\{u\in D_+:u_{\lambda}(z)\leq u(z)\ \mbox{for all}\ z\in\overline{\Omega}\}.
	\end{equation}
	
	On account of (\ref{eq57}) we see that we may assume that
	\begin{equation}\label{eq58}
		K_{\psi^*_{\lambda}}\cap[0,u_{\nu}]=\{u_{\lambda}\}.
	\end{equation}
	
	Indeed, if (\ref{eq58}) is not true, then we have $\hat{u}_{\lambda}\in K_{\psi^*_{\lambda}}\cap[0,u_{\nu}]$, $\hat{u}_{\lambda}-u_{\lambda}\in C_+\backslash\{0\}$, which is a second positive solution of \eqref{eqp} (see (\ref{eq56}), (\ref{eq57})). Moreover, as before, using hypotheses $H(g)(v),H(f)(iv)$ and the strong maximum principle, we have $\hat{u}_{\lambda}-u_{\lambda}\in D_+$ and so we are done.
	
	We introduce the following truncation of $k^*_{\lambda}(z,\cdot)$:
	\begin{equation}\label{eq59}
		\hat{k}^*_{\lambda}(z,x)=\left\{\begin{array}{ll}
			k^*_{\lambda}(z,x)&\mbox{if}\ x<u_{\nu}(z)\\
			k^*_{\lambda}(z,u_{\nu}(z))&\mbox{if}\ u_{\nu}(z)\leq x.
		\end{array}\right.
	\end{equation}
	
	This is a Carath\'eodory function. We set $\hat{K}^*_{\lambda}(z,x)=\int^x_0\hat{k}^*_{\lambda}(z,s)ds$ and consider the $C^1$-functional $\hat{\psi}^*_{\lambda}:H^1(\Omega)\rightarrow\RR$ defined by
	 $$\hat{\psi}^*_{\lambda}(u)=\frac{1}{2}\vartheta(u)+\frac{\mu}{2}||u||^2_2-\int_{\Omega}\hat{k}^*_{\lambda}(z,u)dz\ \mbox{for all}\ u\in H^1(\Omega).$$
	
	As in the proof of Proposition \ref{prop7} we see that
	\begin{equation}\label{eq60}
		K_{\hat{\psi}^*_{\lambda}}\subseteq[u_{\lambda},u_{\nu}]\cap D_+\ (\mbox{see (\ref{eq57})}).
	\end{equation}
	
	By (\ref{eq2}) and (\ref{eq59}) it is clear that $\hat{\psi}^*_{\lambda}$ is coercive. Also, it is sequentially  weakly lower semicontinuous. So, we can find $u^*_{\lambda}\in H^1(\Omega)$ such that
	\begin{eqnarray*}
		 \hat{\psi}^*_{\lambda}(u^*_{\lambda})=\inf[\hat{\psi}^*_{\lambda}(u):u\in H^1(\Omega)],
		&\Rightarrow&u^*_{\lambda}\in[u_{\lambda},u_{\nu}]\cap D_+\ (\mbox{see (\ref{eq60})}).
	\end{eqnarray*}
	
	Note that $\left.(\hat{\psi}^*_{\lambda})'\right|_{[0,u_{\nu}]}=\left.(\psi^*_{\lambda})'\right|_{[0,u_{\nu}]}$ (see (\ref{eq56}), (\ref{eq59})). Therefore
	\begin{eqnarray*}
		 u^*_{\lambda}\in K_{\psi^*_{\lambda}},
		&\Rightarrow&u^*_{\lambda}=u_{\lambda}\ (\mbox{see (\ref{eq58})}).
	\end{eqnarray*}
	
	Moreover, reasoning as in the proof of Corollary \ref{cor8}, we show that
	\begin{eqnarray*}
		&&u_{\nu}-u_{\lambda}\in D_+,\ u_{\lambda}\in D_+,\\
		&\Rightarrow&u_{\lambda}\ \mbox{is a local}\ C^1(\overline{\Omega})-\mbox{minimizer of}\ \psi^*_{\lambda},\\
		&\Rightarrow&u_{\lambda}\ \mbox{is a local}\ H^1(\Omega)-\mbox{minimizer of}\ \psi^*_{\lambda}\ (\mbox{see Proposition \ref{prop4}}).
	\end{eqnarray*}
	
	We can assume that $K_{\psi^*_{\lambda}}$ is finite (otherwise on account of (\ref{eq57}) we see that we already have an infinity of positive smooth solutions strictly bigger than $u_{\lambda}$).
	
	Since $K_{\psi^*_{\lambda}}$ is finite, we can find $\rho\in(0,1)$ small such that
	\begin{equation}\label{eq61}
		 \psi^*_{\lambda}(u_{\lambda})<\inf[\psi^*_{\lambda}(u):||u-u_{\lambda}||=\rho]=m^{\lambda}_{\rho}
	\end{equation}
	(see Aizicovici, Papageorgiou and Staicu \cite{1}, proof of Proposition 29).
	
	Due to hypothesis $H(f)(ii)$ and since $G\geq 0$, we have
	\begin{equation}\label{eq62}
		\psi^*_{\lambda}(t\hat{u}_1)\rightarrow-\infty\ \mbox{as}\ t\rightarrow+\infty\,.
	\end{equation}
	
	Since $k^*_{\lambda}(z,\cdot)$ and $k_{\lambda}(z,\cdot)$ coincide on $\left[u_{\lambda}(z)\right)=\{x\in\RR_+:u_{\lambda}(z)\leq x\}$, we infer that
	\begin{equation}\label{eq63}
		\psi^*_{\lambda}\ \mbox{satisfies the C-condition}
	\end{equation}
	(see the proof of Proposition \ref{prop6}).
	
	Then (\ref{eq61}), (\ref{eq62}), (\ref{eq63}) permit the use of Theorem \ref{th3} (the mountain pass theorem). So, there is $\hat{u}_{\lambda}\in H^1(\Omega)$ such that
	\begin{eqnarray*}
		&&\hat{u}_{\lambda}\in K_{\psi^*_{\lambda}}\ \mbox{and}\ m^{\lambda}_{\rho}\leq \psi^*_{\lambda}(\hat{u}_{\lambda}),\\
		&\Rightarrow&u_{\lambda}\leq\hat{u}_{\lambda}\ \mbox{and}\ \hat{u}_{\lambda}\neq u_{\lambda}\ (\mbox{see (\ref{eq57}) and (\ref{eq61})}).
	\end{eqnarray*}
	
	Moreover, as in the proof of Corollary \ref{cor8}, using hypotheses $H(g)(v)$ and $H(f)(iv)$ and the strong maximum principle, we obtain
	$$\hat{u}_{\lambda}-u_{\lambda}\in D_+.$$
\end{proof}
\begin{proposition}\label{prop11}
	If hypotheses $H(\xi),H(\beta),H(g),H(f)$ and $H_0$ hold, then $\lambda^*\in\mathcal{L}$.
\end{proposition}
\begin{proof}
	Let $\{\lambda_n\}_{n\geq 1}\subseteq(0,\lambda^*)$ such that $\lambda_n\uparrow\lambda^*$. As in the second half of the proof of Proposition \ref{prop7} (see the part of that proof after (\ref{eq47})), we can find $\{u_n\}_{n\geq 1}\subseteq D_+$ such that
	\begin{align}
		&\left\langle A(u_n),h\right\rangle+\int_{\Omega}\xi(z)u_nhdz+\int_{\partial\Omega}\beta(z)u_nhd\sigma=\int_{\Omega}[\lambda_ng(z,u_n)+f(z,u_n)]hdz\label{eq64}\\
		&\mbox{for all}\ h\in H^1(\Omega)\ \mbox{and all}\ n\in\NN,\nonumber\\
		&\hat{\varphi}_{\lambda_n}(u_n)<0\ \mbox{for all}\ n\in\NN\,.\label{eq65}
	\end{align}
	
	In (\ref{eq64}) we choose $h=u_n\in H^1(\Omega)$. Then
	\begin{equation}\label{eq66}
		\vartheta(u_n)=\int_{\Omega}[\lambda_ng(z,u_n)+f(z,u_n)]u_ndz\ \mbox{for all}\ n\in\NN\,.
	\end{equation}
	
	By (\ref{eq65}) we have
	\begin{equation}\label{eq67}
		\vartheta(u_n)-2\int_{\Omega}[\lambda_n G(z,u_n)+F(z,u_n)]dz<0\ \mbox{for all}\ n\in\NN\ (\mbox{see (\ref{eq5})}).
	\end{equation}
	
	It follows from (\ref{eq66}) and (\ref{eq67}) that
	\begin{eqnarray}\label{eq68}
		&&\int_{\Omega}\gamma_{\lambda_n}(z,u_n)dz<0\ \mbox{for all}\ n\in\NN,\nonumber\\
		&\Rightarrow&\int_{\Omega}\gamma_{\lambda^*}(z,u_n)dz<0\ \mbox{for all}\ n\in\NN\ (\mbox{see hypothesis}\ H(g)(iv)).
	\end{eqnarray}
	
	Then reasoning as in the proof of Proposition \ref{prop6} (see the claim) and applying (\ref{eq68}), we show that $\{u_n\}_{n\geq 1}\subseteq H^1(\Omega)$ is bounded. So, we may assume that
	\begin{equation}\label{eq69}
		u_n\stackrel{w}{\rightarrow}u^*\ \mbox{in}\ H^1(\Omega)\ \mbox{and}\ u_n\rightarrow u^*\ \mbox{in}\ L^{2s'}(\Omega)\ \mbox{and in}\ L^2(\partial\Omega).
	\end{equation}
	
	In (\ref{eq64}) we choose $h=u_n-u^*\in H^1(\Omega)$, pass to the limit as $n\rightarrow\infty$ and use (\ref{eq69}). Then
	\begin{eqnarray}\label{eq70}
		&&\lim\limits_{n\rightarrow\infty}\left\langle A(u_n),u_n-u^*\right\rangle=0,\nonumber\\
		&\Rightarrow&u_n\rightarrow u^*\ \mbox{in}\ H^1(\Omega)\ (\mbox{by the Kadec-Klee property, see (\ref{eq69})}).
	\end{eqnarray}
	
	So, if in (\ref{eq64}) we pass to the limit as $n\rightarrow\infty$ and use (\ref{eq70}), we infer that
	$$u^*\in C_+\ \mbox{is a nonnegative solution of}\ (P_{\lambda^*}).$$
	
	If we show that $u^*\neq 0$, then we are finished. To this end, note that we can find $c_{15}>0$ such that
	\begin{equation}\label{eq71}
		\lambda g(z,x)+f(z,x)>\lambda c_3x^{q-1}-c_{15}x\ \mbox{for almost all}\ z\in\Omega\ \mbox{and all}\ x\geq 0
	\end{equation}
	(see hypothesis $H(g)(iii)$ and hypotheses $H(f)(i),(ii),(iii)$). Let $\lambda=\lambda_1\leq \lambda_n$ for all $n\in\NN$ and consider the following auxiliary Robin problem
	\begin{equation}\label{eq72}
		\left\{\begin{array}{ll}
			-\Delta u(z)+\xi(z)u(z)=\lambda_1c_3u(z)^{q-1}-c_{12}u(z)&\mbox{in}\ \Omega,\\
			\dfrac{\partial u}{\partial n}+\beta(z)u=0&\mbox{on}\ \partial\Omega,\ u>0.
		\end{array}\right\}
	\end{equation}
	
	Let $d:H^1(\Omega)\rightarrow\RR$ be the $C^1$-functional defined by
	 $$d(u)=\frac{1}{2}\vartheta(u)+\frac{\mu}{2}||u^-||^2_2+\frac{c_{15}}{2}||u^+||^2_2-\frac{\lambda_1c_3}{q}||u^+||^q_q\ \mbox{for all}\ u\in H^1(\Omega).$$
	
	Using (\ref{eq2}) and the fact that $q<2$, we infer that $d(\cdot)$ is coercive. Also, it is sequentially weakly lower semicontinuous. So, we can find $\bar{u}\in H^1(\Omega)$ such that
	\begin{equation}\label{eq73}
		d(\bar{u})=\inf[d(u):u\in H^1(\Omega)].
	\end{equation}
	
	Since $q<2$, for $t\in(0,1)$ small enough, we have
	\begin{eqnarray*}
		 d(t\hat{u}_1)<0,
		&\Rightarrow&d(\bar{u})<0=d(0)\ (\mbox{see (\ref{eq73})}),\\
		&\Rightarrow&\bar{u}\neq 0.
	\end{eqnarray*}
	
	By (\ref{eq73}) we have
	\begin{eqnarray}\label{eq74}
		&&d'(\bar{u})=0,\nonumber\\
		&\Rightarrow&\left\langle A(\bar{u}),h\right\rangle+\int_{\Omega}\xi(z)\bar{u}hdz+\int_{\partial\Omega}\beta(z)\bar{u}hd\sigma-\int_{\Omega}\mu\bar{u}^-hdz\nonumber\\
		&&=\int_{\Omega}[\lambda_1c_3(\bar{u}^+)^{q-1}-c_{15}(\bar{u}^+)]hdz\ \mbox{for all}\ h\in H^1(\Omega).
	\end{eqnarray}
	
	In (\ref{eq74}) we choose $h=-\bar{u}^-\in H^1(\Omega)$. Then
	\begin{eqnarray*}
		 \vartheta(\bar{u}^-)+\mu||\bar{u}^-||^2_2=0,
		&\Rightarrow&c_0||\bar{u}^-||^2\leq 0\ (\mbox{see (\ref{eq2})}),\\
		&\Rightarrow&\bar{u}\geq 0,\ \bar{u}\neq 0.
	\end{eqnarray*}
	
	Then (\ref{eq74}) becomes
	\begin{eqnarray*}
		&&\left\langle A(\bar{u}),h\right\rangle+\int_{\Omega}\xi(z)\bar{u}hdz+\int_{\partial\Omega}\beta(z)\bar{u}hd\sigma=\int_{\Omega}[\lambda_1c_3\bar{u}^{q-1}-c_{15}\bar{u}]hdz\\
		&&\mbox{for all}\ h\in H^1(\Omega),\\
		 &\Rightarrow&-\Delta\bar{u}(z)+\xi(z)\bar{u}(z)=\lambda_1c_3\bar{u}(z)^{q-1}-c_{15}\bar{u}(z)\ \mbox{for almost all}\ z\in\Omega,\\
		&&\frac{\partial \bar{u}}{\partial n}+\beta(z)\bar{u}=0\ \mbox{on}\ \partial\Omega\ (\mbox{see Papageorgiou and R\u adulescu \cite{11}})\\
		&\Rightarrow&\bar{u}\ \mbox{is a positive solution of (\ref{eq72})}.
	\end{eqnarray*}
	
	Moreover, using the regularity results of Wang \cite{16} and the strong maximum principle, we have
	$$\bar{u}\in D_+.$$
	
	Recall that $u_n\in D_+$ for all $n\in\NN$. So, we can find $\eta_n>0$ such that $\eta_n\bar{u}\leq u_n$. We choose $\eta_n$ to be the biggest such positive real and suppose that $\eta_n\in(0,1)$. Also, let $\xi^*_n>c_{15}>0$. Then
	\begin{align}\label{eq75}
		&-\Delta(\eta_n\bar{u})+(\xi(z)+\xi^*_n)(\eta_n\bar{u})\nonumber\\
		=&\eta_n[-\Delta\bar{u}+(\xi(z)+\xi^*_n)\bar{u}]\nonumber\\
		=&\eta_n[\lambda_1c_3\bar{u}^{q-1}+(\xi^*_n-c_{15})\bar{u}]\ (\mbox{see (\ref{eq72})})\nonumber\\
		\leq&\lambda_1c_3(\eta_n\bar{u})^{q-1}+(\xi^*_n-c_{15})(\eta_n\bar{u})\ (\mbox{recall that}\ \eta_n\in(0,1)\ \mbox{while}\ q<2)\nonumber\\
		\leq&\lambda_1c_3u_n^{q-1}+(\xi^*_n-c_{15})u_n\ (\mbox{recall that}\ \eta_n\bar{u}\leq u_n\ \mbox{and}\ \xi^*_n-c_{15}>0)\nonumber\\
		<&\lambda_ng(z,u_n)+f(z,u_n)+\xi^*_nu_n\ (\mbox{see (\ref{eq71}) and recall that}\ \lambda_1\leq\lambda_n\ \mbox{for all}\ n\in\NN)\nonumber\\
		=&-\Delta u_n+(\xi(z)+\xi^*_n)u_n\ (\mbox{since}\ u_n\in S^{\lambda_n}_+),\nonumber\\
		 \Rightarrow&\Delta(u_n-\eta_n\bar{u})\leq(||\xi^+||_{\infty}+\xi^*_n)(u_n-\eta_n\bar{u})\ (\mbox{see hypothesis}\ H(\xi)).
	\end{align}
	
	Evidently, $u_n\neq\eta_n\bar{u}$. So, from (\ref{eq75}) and the strong maximum principle, we infer that
	$$u_n-\eta_n\bar{u}\in D_+,$$
	which contradicts the maximality of $\eta_n$. Hence $\eta_n\geq 1$ and so
	\begin{eqnarray*}
		&&\bar{u}\leq u_n\ \mbox{for all}\ n\in\NN,\\
		&\Rightarrow&\bar{u}\leq u^*\ (\mbox{see (\ref{eq70})}),\\
		&\Rightarrow&u^*\neq 0\ \mbox{and so}\ u^*\in S^{\lambda^*}_+\subseteq D_+,\ \mbox{thus}\ \lambda^*\in\mathcal{L}.
	\end{eqnarray*}
\end{proof}

 This proposition implies that
$$\mathcal{L}=\left(0,\lambda^*\right].$$

\section{Extremal positive solutions - bifurcation theorem}

In this section,  we first show that for every $\lambda\in(0,\lambda^*)$ problem \eqref{eqp} has a smallest positive solution $\tilde{u}_{\lambda}\in D_+$ and determine the monotonicity and continuity properties of the map $\lambda\mapsto\tilde{u}_{\lambda}.$
\begin{proposition}\label{prop12}
	If hypotheses $H(\xi),H(\beta),H(g),H(f)$ and $H_0$ hold, then for every $\lambda\in(0,\lambda^*)$, problem \eqref{eqp} has a smallest positive solution $\tilde{u}_{\lambda}\in D_+$ and the map $\lambda\mapsto\tilde{u}_{\lambda}$ is strictly increasing in the sense that
	$$\tau<\lambda\Rightarrow\tilde{u}_{\lambda}-\tilde{u}_{\tau}\in D_+$$
	and it is left continuous from $(0,\lambda^*)$ into $C^1(\overline{\Omega})$.
\end{proposition}
\begin{proof}
	As in Filippakis and Papageorgiou \cite[Lemma 4.1]{5}, we have that $S^{\lambda}_+$ is downward directed (that is, if $u_1,u_2\in S^{\lambda}_+$, then we can find $u\in S^{\lambda}_+$ such that $u\leq u_1,u\leq u_2$). Invoking Lemma 3.10 of Hu and Papageorgiou \cite[p. 178]{7}, we can find a decreasing sequence $\{u_n\}_{n\geq 1}\subseteq S^{\lambda}_+$  such that
	$$\inf S^{\lambda}_+=\inf\limits_{n\geq 1}u_n.$$
	
	We may assume that
	\begin{equation}\label{eq76}
		u_n\stackrel{w}{\rightarrow}\tilde{u}_{\lambda}\ \mbox{in}\ H^1(\Omega)\ \mbox{and}\ u_n\rightarrow\tilde{u}_{\lambda}\ \mbox{in}\ L^{2s'}(\Omega)\ \mbox{and in}\ L^2(\partial\Omega).
	\end{equation}
	
	We have
	\begin{align}\label{eq77}
		&\left\langle A(u_n),h\right\rangle+\int_{\Omega}\xi(z)u_nhdz+\int_{\partial\Omega}\beta(z)u_nhd\sigma=\int_{\Omega}[\lambda g(z,u_n)+f(z,u_n)]hdz\nonumber\\
		&\mbox{for all}\ h\in H^1(\Omega),\nonumber\\
		\Rightarrow&\left\langle A(\tilde{u}_{\lambda}),h\right\rangle+\int_{\Omega}\xi(z)\tilde{u}_{\lambda}hdz+\int_{\partial\Omega}\beta(z)\tilde{u}_{\lambda}hd\sigma=\int_{\Omega}[\lambda g(z,\tilde{u}_{\lambda})+f(z,\tilde{u}_{\lambda})]hdz\\
		&\mbox{for all}\ h\in H^1(\Omega)\ (\mbox{see (\ref{eq76})}).\nonumber
	\end{align}
	
	Also, by the proof of Proposition \ref{prop11} and since $\lambda_1<\lambda$ (see equation (\ref{eq72})), we have
	\begin{eqnarray*}
		 \bar{u}\leq u_n\ \mbox{for all}\ n\in\NN,
		&\Rightarrow&\bar{u}\leq\tilde{u}_{\lambda}\ (\mbox{see (\ref{eq76})}),\\
		&\Rightarrow&\tilde{u}_{\lambda}\neq 0\ \mbox{and so}\ \tilde{u}_{\lambda}\in S^{\lambda}_+,\ \tilde{u}_{\lambda}=\inf S^{\lambda}_+.
	\end{eqnarray*}
	
	If $\tau<\lambda$, then by Corollary \ref{cor8} we can find $u_{\tau}\in S^{\tau}_{\lambda}$ such that
	\begin{eqnarray}\label{eq78}
		 \tilde{u}_{\lambda}-u_{\tau}\in D_+,
		&\Rightarrow&\tilde{u}_{\lambda}-\tilde{u}_{\tau}\in D_+,\nonumber\\
		&\Rightarrow&\lambda\mapsto\tilde{u}_{\lambda}\ \mbox{is strictly increasing}.
	\end{eqnarray}
	
	Finally, suppose that $\lambda_n\rightarrow \lambda^-\ (\lambda\in(0,\lambda^*))$. From the regularity theory (see Wang \cite{16}), we know that we can find $\alpha\in(0,1)$ and $c_{16}>0$ such that
	$$\tilde{u}_{\lambda_n}\in C^{1,\alpha}(\overline{\Omega}),\ ||\tilde{u}_{\lambda_n}||_{C^{1,\alpha}(\overline{\Omega})}\leq c_{16}\ \mbox{for all}\ n\in\NN.$$
	
	Exploiting the compact embedding of $C^{1,\alpha}(\overline{\Omega})$ into $C^1(\overline{\Omega})$ and by passing to a subsequence if necessary, we have that
	\begin{equation}\label{eq79}
		u_n\rightarrow\bar{u}_{\lambda}\ \mbox{in}\ C^1(\overline{\Omega}).
	\end{equation}
	
	Suppose that $\bar{u}_{\lambda}\neq\tilde{u}_{\lambda}$. Then we can find $z_0\in\Omega$ such that
	\begin{eqnarray*}
		 \tilde{u}_{\lambda}(z_0)<\bar{u}_{\lambda}(z_0),
		&\Rightarrow&\tilde{u}_{\lambda}(z_0)<\tilde{u}_{\lambda_n}(z_0)\ \mbox{for all}\ n\geq n_0\ (\mbox{see (\ref{eq79})}),
	\end{eqnarray*}
	which contradicts (\ref{eq78}) (recall that $\lambda_n\leq \lambda$ for all $n\in\NN$). Therefore by the Urysohn criterion,  we have for the original sequence
	\begin{eqnarray*}
		 \tilde{u}_{\lambda_n}\rightarrow\tilde{u}_{\lambda}\ \mbox{in}\ C^1(\overline{\Omega}),
		&\Rightarrow&\lambda\mapsto\tilde{u}_{\lambda}\ \mbox{is left continuous from}\ (0,\lambda^*)\ \mbox{into}\ C^1(\overline{\Omega}).
	\end{eqnarray*}
\end{proof}

Summarizing the results of Sections 3 and 4, we can formulate the following bifurcation-type result, describing the behavior of the set of positive solutions of \eqref{eqp} with respect to the parameter $\lambda>0$.
\begin{theorem}\label{th13}
	If hypotheses $H(\xi),H(\beta),H(g),H(f)$ and $H_0$ hold, then there exists $\lambda^*>0$ such that
	\begin{itemize}
		\item[(a)] for every $\lambda\in(0,\lambda^*)$ problem \eqref{eqp} has at least two positive solutions
		$$u_{\lambda},\hat{u}_{\lambda}\in D_+\ \mbox{with}\ \hat{u}_{\lambda}-u_{\lambda}\in D_+,$$
		it has a smallest positive solution $\tilde{u}_{\lambda}\in D_+$ and the map $\lambda\mapsto\tilde{u}_{\lambda}$ from $(0,\lambda^*)$ into $C^1(\overline{\Omega})$ is strictly increasing in the sense that
			$$\tau<\lambda\Rightarrow\tilde{u}_{\lambda}-\tilde{u}_{\tau}\in D_+$$
			and is left continuous;
		\item[(b)] for $\lambda=\lambda^*$ problem \eqref{eqp} has at least one positive solution $u^*\in D_+$;
		\item[(c)] for $\lambda>\lambda^*$ problem \eqref{eqp} has no positive solutions.
	\end{itemize}
\end{theorem}

\section*{Acknowledgments} This research  was supported  by  the  Slovenian  Research  Agency
grants P1-0292, J1-7025, and J1-6721, and the Romanian National
Authority for Scientific Research and Innovation, CNCS-UEFISCDI, project  PN-III-P4-ID-PCE-2016-0130.
We thank the referee for comments.


\medskip 

\end{document}